\documentclass[a4paper,12pt]{amsart}
\usepackage{amssymb}
\usepackage{amsmath}
\usepackage{stmaryrd}

\def\a{\alpha}
\def\b{\beta}

\def\e{\varepsilon}

\def\s{\sigma}

\def\L{\Lambda}

\def\bb{\mathrm {b}}

\def\beq{\begin{equation}}
\def\eeq{\end{equation}}
\def\bea{\begin{eqnarray*}}
\def\eea{\end{eqnarray*}}
\def\beaa{\begin{eqnarray}}
\def\eeaa{\end{eqnarray}}
\def\ba{\begin{array}}
\def\ea{\end{array}}
\def\bp{\begin{proof}}
\def\r{\end{proof}}

\def\la{\langle}
\def\ra{\rangle}
\def\.{\cdot}

\def \RM{\mathbb{R}}

\def \CM{\mathbb{C}}
\def \KM{\mathbb{K}}
\def \HM{\mathbb{H}}
\def \SM{\mathbb{S}}

\def\W{{\mathcal W}}
\def\R{{\mathcal R}}

\def\id{\mathrm{id}}
\def\tr{\mathrm{tr}}

\def\Sym{\mathrm{Sym}}

\def\S{\mathrm{S}}
\def\U{\mathrm{U}}
\def\E{\mathrm{E}}
\def\F{\mathrm{F}}

\def\End{\mathrm{End}}

\def\Cl{\mathrm{Cl}}

\def\Spin{\mathrm{Spin}}

\def\ad{\mathrm{ad}}
\def\dim{\mathrm{dim}}

\def\Sym{\mathrm{Sym}}

\def\Cas{\mathrm{Cas}}

\def\so{\mathfrak{so}}
\def\su{\mathfrak{su}}
\def\spin{\mathfrak{spin}}
\def\sp{\mathfrak{sp}}
\def\ee{\mathfrak{e}}
\def\ff{\mathfrak{f}}
\def\gg{\mathfrak{g}}
\def\hh{\mathfrak{h}}

\def\kk{\mathfrak{k}}
\def\mm{\mathfrak{m}}

\def\uu{\mathfrak{u}}

 \def\ii{\mathfrak{i}}
\def\jj{\mathfrak{j}}

\newtheorem{epr}{Proposition}[section]

\newtheorem{ath}[epr]{Theorem}

\newtheorem{elem}[epr]{Lemma}

\newtheorem{ecor}[epr]{Corollary}

\theoremstyle{definition}

\newtheorem{ede}[epr]{Definition}
\newtheorem{ere}[epr]{Remark}

\title{Invariant four-forms and symmetric pairs}

\author{Andrei Moroianu, Uwe Semmelmann}

\address{Andrei Moroianu \\ CMLS\\ {\'E}cole
  Polytechnique \\ UMR 7640 du CNRS
\\ 91128 Palaiseau \\ France}
\email{am@math.polytechnique.fr}

\address{Uwe Semmelmann\\
Institut f\"ur Geometrie und Topologie \\
Fachbereich Mathematik\\
Universit{\"a}t Stuttgart\\
Pfaffenwaldring 57 \\
70569 Stuttgart, Germany
}
\email{uwe.semmelmann@mathematik.uni-stuttgart.de}

\date{\today}
\thanks{This work was supported through the program "Research in Pairs" by the Mathematisches Forschungsinstitut Oberwolfach. 
We thank the institute for the hospitality and the stimulating research environment. }

\begin{document}

\begin{abstract} We give criteria for real, complex and quaternionic representations to define
$s$-representations, focusing on exceptional Lie algebras defined by spin representations.  As applications,
we obtain the classification of complex representations whose second exterior power is irreducible or has an irreducible
summand of co-dimension one, and we give a conceptual computation-free argument for the construction of 
the exceptional Lie algebras of compact type.

\bigskip


\end{abstract}

\maketitle

\section{Introduction}

The initial impetus for this text was given by an observation by John Baez \cite[p. 200]{baez} in
his celebrated paper {\em The Octonions}, concerning the construction of the exceptional Lie algebra $\ee_8$ 
as the direct sum of $\spin(16)$ and the real half-spin representation $\Sigma^+_{16}$ of $\Spin(16)$. 
After showing that the verification of the Jacobi identity 
reduces to the case where all three vectors lie in $\Sigma^+_{16}$, he writes

{\em``...unfortunately, at this point a brute-force calculation seems to be required. For two approaches that minimize 
the pain, see the books by Adams \cite{adams} and by Green, Schwartz and Witten \cite{gsw}. It would be nice to find a 
more conceptual approach."}

This problem can be rephrased in terms of so-called $s$-representations, introduced by Kostant \cite{ko56} and studied by 
Heintze and Ziller \cite{hz} or Eschenburg and Heintze \cite{eh99} among others. Roughly speaking, an $s$-representation 
is a real representation of some Lie algebra of compact type $\hh$ which can be realized as the isotropy representation of
a symmetric space. It turns out that the obstruction for a given representation $\mm$ to be an $s$-representation is encoded in some
invariant element in the fourth exterior power of $\mm$, defined by the image of the Casimir element from the 
universal enveloping algebra of $\hh$. 

In particular this obstruction automatically vanishes when $(\Lambda^4\mm)^\hh=0$. Now, this condition could seem {\em a priori} 
quite restrictive.
For instance, it can never hold if the representation carries a complex (or all the more quaternionic) structure. This 
is due to the existence of 
universal elements in  $(\Lambda^4\mm)^\hh$ which are inherent to the structure of $\mm$. Nevertheless, if these are the 
only invariant elements, one can adapt our construction by adding a $\uu(1)$ or $\sp(1)$ summand to $\hh$, depending
on whether $\mm$ is complex or quaternionic, in order to ``kill" the obstruction given by the Casimir element of $\hh$,
and turn $\mm$ into an $s$-representation of $\hh\oplus\uu(1)$ or $\hh\oplus\sp(1)$ respectively. These results are summarized
in Propositions \ref{complex} and \ref{quaternionic} below.

The idea of adding a summand to $\hh$ in order to obtain $s$-representations already appears in \cite{hz},
in a setting which presents many similarities with ours. However, the criteria for $s$-representations obtained 
by Heintze and Ziller are somewhat complementary to ours. In the complex setting, for example, Theorem 2 in \cite{hz}
can be stated as follows: If $\mm$ is a faithful complex representation of $\hh$ such that 
the orthogonal complement of $\hh\subset \llbracket \Lambda^{1,1} \mm \rrbracket$ is irreducible, then
$\mm$ is an $s$-representation with respect to $\hh\oplus\uu(1)$. In contrast, in Theorem \ref{com} we prove
a similar statement, but under the different assumption that $\Lambda^2\mm$ is irreducible.

As applications of our results we then obtain in Theorem \ref{com} a geometrical proof of the classification by 
Dynkin of complex representations with irreducible second exterior power
as well as a classification result in Theorem \ref{33} for representations with quaternionic structure whose second exterior power
decomposes in only two irreducible summands. This classification is based on a correspondence between such representations and 
$s$-representations already pointed out by Wang and Ziller in \cite[p. 257]{wz84} where a framework relating symmetric spaces and 
strongly isotropy irreducible homogeneous spaces is introduced. These results can also be compared to the classification by Calabi 
and Vesentini of
complex representations whose symmetric power has exactly two irreducible components based on the classification
of symmetric spaces, reproved by Wang and Ziller \cite{wz93} using representation theoretical methods.

In the last section we apply the above ideas in order to give a purely conceptual proof for the existence of $\ee_8$ in 
Proposition \ref{41}. The same method, using spin representations, also works for the other exceptional simple Lie algebras 
except $\gg_2$. Conversely, we show that most spin representations which are isotropy representations of equal rank 
homogeneous spaces are actually $s$-representations and define the
exceptional Lie algebras. Note that an alternative geometrical approach to the construction of $\ff_4$ and $\ee_8$ using the so-called Killing
superalgebra was recently proposed by J. Figueroa-O'Farrill \cite{of08}.

{\sc Acknowledgments.} We are grateful to Jost-Hinrich Eschenburg and Wolfgang Ziller for having brought to our attention 
the rich literature on $s$-representations and to Steffen Koenig for his helpful remarks about some classification results 
in representation theory. Special thanks are due to John~Baez whose questions motivated our work and who pointed out other related references.

\section{A characterization of $s$-representations}

Let $(\hh,B_\hh)$ be a real Lie algebra of compact type endowed 
with an $\ad_\hh$-invariant Euclidean product
and let $\rho:\hh\to
\End(\mm)$ be a faithful irreducible
representation of $\hh$ over $\RM$,
endowed with an invariant Euclidean product $B_{\mm}$ (defined up to some positive constant). 
In order to simplify the notation we will 
denote $\rho(a)(v)$ by $av$ for all $a\in\hh$ and $v\in \mm$. 
Our first goal is to find necessary and sufficient conditions for
the existence of a Lie algebra structure on $\gg:=\hh\oplus\mm$
compatible with the above data. 

\begin{elem} \label{l1} There exists a unique $(\RM$-linear$)$ bracket $[.,.]:\Lambda^2\gg^*\to\gg$ on
  $\gg:=\hh\oplus\mm$ such that
\begin{enumerate}
\item  $[.,.]$ restricted to $\hh$ is the usual Lie algebra bracket.
\item $[a,v]=-[v,a]=av$ for all $a\in\hh$ and $v\in \mm$. 
\item $[\mm,\mm]\subset\hh$.
\item $B_{\hh}(a,[v,w])=B_{\mm}(av,w)$  for all $a\in\hh$ and $v,w\in \mm$. 
\end{enumerate}
\end{elem}
\bp
The uniqueness is clear. For the existence we just need to check that the restriction 
of $[.,.]$ to $\mm\otimes\mm$ given by (4) is skew-symmetric. This follows from
the $\ad_{\hh}$-invariance of $B_{\mm}$.
\r

\begin{ede}({\em cf.} \cite{ko56}) An irreducible representation $\mm$ of a normed Lie algebra
  $(\hh,B_\hh)$ such that the bracket given by Lemma \ref{l1} defines 
a Lie algebra structure on $\gg:=\hh\oplus\mm$ is called an {\em
  $s$-representation}. The Lie algebra $\gg$ is called the {\em augmented}
  Lie algebra of the $s$-representation $\mm$. 
\end{ede}

Note that the above construction was studied in greater generality by
Kostant \cite{ko}, who introduced the notion of {\em orthogonal
representation of Lie type}, satisfying conditions (1), (2) and (4)
in Lemma \ref{l1}. One can compare his characterization of representations of Lie
type (\cite{ko}, Thm. 1.50 and 1.59) with Proposition \ref{cas} below.

\begin{ere} If $\gg$ is the augmented Lie algebra of an
 $s$-representation of $(\hh,B_\hh)$ on $\mm$, then the involution
  $\s:=\id_\hh-\id_\mm$ is an automorphism of $\gg$, and $(\gg,\hh,\s)$ is a symmetric
  pair of compact type. Conversely, every irreducible symmetric
  pair of compact type can be obtained in this way.
\end{ere}

In the sequel $\{e_i\}$ will denote a $B_\mm$-orthonormal basis
of $\mm$ and $\{a_k\}$  a $B_\hh$-orthonormal basis
of $\hh$.

\begin{elem}\label{cs}
If an irreducible $s$-representation $(\mm,B_\mm)$ of $\hh$ has a $\hh$-invariant orthogonal complex structure, then 
$\hh$ is not semi-simple.
\end{elem}
\bp
If $J$ denotes the complex structure of $\mm$ commuting with the $\hh$-action,
then for all $\a\in\hh$ and $v,w\in\mm$ we have
$$B_\hh(a,[Jv,w])=B_\mm(aJv,w)=B_\mm(Jav,w)=-B_\mm(av,Jw)=-B_\hh(a,[v,Jw]),$$
whence 
\beq\label{j}[Jv,w]=-[v,Jw].\eeq 
We now consider the element 
\beq\label{daj}a_J:=\sum_i[e_i,Je_i]\in\hh\eeq
which clearly belongs to the center of $\hh$.  In order to show that it does not vanish, we compute using the Jacobi identity
$$a_Jv=[a_J,v]=\sum_i[[e_i,Je_i],v]=-\sum_i([[Je_i,v],e_i]+[[v,e_i],Je_i])=-2\sum_i[[Je_i,v],e_i].$$
We take $v=e_j$, make the scalar product with $Je_j$ and sum over $j$. Using \eqref{j} we get
\beq\label{aj}\sum_jB_\mm(a_Je_j,Je_j)=-2\sum_{i,j}B_\hh([Je_i,e_j],[e_i,Je_j])=2\sum_{i,j}B_\hh([Je_i,e_j],[Je_i,e_j]).\eeq
If $a_J=0$ this equation would imply that the bracket vanishes on $\mm$, whence
$$0=B_\hh(a,[v,w])=B_\mm(av,w),\qquad \forall a\in\hh,\ \forall v,w\in\mm,$$
which is clearly impossible.
\r

The element $a_J$ defined in the proof above plays an important r\^ole in the theory of
Hermitian symmetric spaces. For now, let us remark that because of the irreducibility of $\mm$ and of the fact that
$a_J$ belongs to the center of $\hh$, there exists a non-zero constant $\mu$, depending on the choice of $B_\mm$, such that
\beq\label{aj1} a_Jv=\mu Jv, \qquad\forall v\in\mm,
\eeq
in other words $a_J$ acts like $\mu J$ on $\mm$.

Using the $\hh$-invariant scalar product $B_\mm$, the representation $\rho$ 
induces a Lie algebra morphism $\tilde\rho:\hh\to\L^2\mm\subset \L^{even}\mm$,
$a\mapsto\tilde\rho(a):=\tilde a$ where 
\beq\label{rho}
\tilde a(u,v)=B_\mm(au,v)=B_\hh(a,[u,v]).
\eeq

For later use, we recall that the induced 
Lie algebra action of $\hh$ on exterior 2-forms is $a_*(\tau)(u,v):=-\tau(au,v)-\tau(u,av)$,
for all $\tau\in\Lambda^2\mm$,
whence, in particular, the following formula holds:
\beq\label{ind} a_*\tilde b=\widetilde{[a,b]},\qquad\forall a,b\in\hh.
\eeq

The Lie algebra morphism $\tilde\rho$ extends to an
algebra morphism $\tilde\rho:\U(\hh)\to \L^{even}\mm$, where $\U(\hh)$
denotes the enveloping algebra of $\hh$. This morphism maps
the Casimir element $\Cas_\hh=\sum_k (a_k)^2$ of $\hh$ to an invariant element $\tilde\rho(\Cas_\hh)\in(\L^4\mm)^\hh$. 
It was remarked by Kostant \cite{ko56} and Heintze and Ziller \cite{hz} that this element is exactly the obstruction for 
$\mm$ to be an $s$-representation. We provide the proof of this fact below for the reader's convenience.

\begin{epr}[\cite{hz}, \cite{ko56}] \label{cas} 
An irreducible representation $(\mm,B_\mm)$ of $\hh$ is an $s$-representation if and only if
$\tilde\rho(\Cas_\hh)=0$.
\end{epr}
\bp We need to check that the Jacobi identity for the bracket defined in Lemma \ref{l1} on $\hh\oplus\mm$
is equivalent to the vanishing of $\tilde\rho(\Cas_\hh)$. Note that the Jacobi identity is 
automatically satisfied by $[.,.]$ whenever one of the three entries belongs to $\hh$.

We now take four arbitrary vectors $u,v,w,z\in\mm$ and compute the obstruction 
$$\mathcal{J}(u,v,w):=[[u,v],w]+ [[v,w],u]+ [[w,u],v]$$
using \eqref{rho} as follows:
\bea B_\mm(\mathcal{J}(u,v,w),z)&=&B_\mm([[u,v],w]+ [[v,w],u]+ [[w,u],v],z)\\
&=&B_\hh([u,v],[w,z])+B_\hh([v,w],[u,z])+B_\hh([w,u],[v,z])\\
&=&\sum_k\big(B_\hh(a_k,[u,v])B_\hh(a_k,[w,z])+B_\hh(a_k,[v,w])B_\hh(a_k,[u,z])\\
&&\qquad+B_\hh(a_k,[w,u])B_\hh(a_k,[v,z])\big)\\
&=&\sum_k(\tilde a_k(u,v)\tilde a_k(w,z)+\tilde a_k(v,w)\tilde a_k(u,z)+\tilde a_k(w,u)\tilde a_k(v,z))\\
&=&\frac12\sum_k(\tilde a_k\wedge\tilde a_k)(u,v,w,z)=\frac12\tilde\rho(\Cas_\hh)(u,v,w,z).
\eea

\r
The above result yields a simple criterion for $s$-representations:

\begin{ecor}\label{c1} If $(\L^4\mm)^\hh=0$, then $\mm$ is an $s$-representation.
\end{ecor}

Conversely, one could ask whether every $s$-representation arises in this way.
One readily sees that this is not the case, since the condition $(\L^4\mm)^\hh=0$ can only hold if
$\hh$ is simple and $\mm$ is a purely real representation
({\em cf.} Lemma \ref{tensor} below). Nevertheless, under these restrictions, the converse
to Corollary~\ref{c1} also holds, {\em cf.} Proposition \ref{rem} below.

\begin{elem}\label{tensor}
Let $\mm$ be an irreducible real representation of $(\hh,B_\hh)$ with $\dim_\RM(\mm)\ge4$. Then $(\L^4\mm)^\hh$
is non-zero if either $\mm$ has a complex structure or $\hh$ is not simple.

\end{elem}
\bp If $J$ is a $\hh$-invariant complex structure on $\mm$, then $B_\mm(J.,J.)$
is a positive definite $\hh$-invariant scalar product on $\mm$ so by the irreducibility of
$\mm$ there is some positive constant $\nu$ such that $B_\mm(Ju,Jv)=\nu B_\mm(u,v)$
for every $u,v\in\mm$. Applying this relation to $Ju,Jv$ yields $\nu^2=1$, so $\nu=1$,
{\em i.e.} $J$ is orthogonal. The corresponding 2-form $\omega\in\Lambda^2\mm$
defined by 
\beq\label{o}\omega(u,v):=B_\mm(Ju,v)
\eeq
is $\hh$-invariant. Moreover, since $\dim_\RM(\mm)\ge4$, the four-form $\omega\wedge\omega$
is a non-zero element in $(\L^4\mm)^\hh$.

Assume now that $\hh=\hh_1\oplus\hh_2$ is not simple. Then $\mm=\mm_1\otimes\mm_2$
is the tensor product of irreducible representations $\mm_i$ of $\hh_i$. We endow
each $\mm_i$ with an $\hh_i$-invariant scalar product and identify $\mm$ with 
the representation $L(\mm_1,\mm_2)$ of linear maps between $\mm_1$ and $\mm_2$.
If $u\in L(\mm_1,\mm_2)$ we denote by $ u^*\in L(\mm_2,\mm_1)$ its adjoint. We now consider the
element $R\in\Sym^2(\Lambda^2\mm)$ given by 
$$R(u,v,w,z):=\tr((uv^*-vu^*)(wz^*-zw^*)),
$$
and the four-form $\Omega:=\b(R)$, where $\b$ is the Bianchi map $\b:\Sym^2(\Lambda^2(\mm))\to\Lambda^4(\mm)$ defined
by 
$$\b(T)(u,v,w,z):=T(u,v,w,z)+T(v,w,u,z)+T(w,u,v,z).$$
It is clear that $\Omega$ belongs to $(\L^4\mm)^\hh$ (it is actually $\so(\mm_1)\oplus\so(\mm_2)$-invariant).
To see that it is non-zero, take orthonormal bases $\{x_i\}$, $\{y_j\}$ of $\mm_1$ and $\mm_2$
and check that $\Omega(z_{11}, z_{12},z_{21},z_{22})=2$ for $z_{ij}:=x_i\otimes y_j$.
\r

In view of Lemma \ref{tensor}, it would be interesting to relax the condition 
$(\L^4\mm)^\hh=0$ in Corollary~\ref{c1} in order to obtain a criterion which could cover also
the cases of complex or quaternionic representations. Let us first
clarify the terminology. It is well-known that if
$\rho:\hh\to\End(\mm)$ is an irreducible 
$\RM$-representation of $\hh$, the centralizer of $\rho(\hh)$ in
$\End(\mm)$ is an algebra isomorphic to $\RM$, $\CM$ or
$\HM$. Correspondingly, we will say that $\mm$ has real, complex or
quaternionic type respectively.

Remark that if a real representation $\mm$ of a semi-simple Lie algebra $(\hh,B_\hh)$ of compact type
has a complex structure $I$, then it can not be an $s$-representation
by Lemma \ref{cs}. 
Nevertheless, it turns out that the natural extension of $\rho$
to $\hh\oplus\uu(1)$ defined on the generator $\ii\in\uu(1)$ by $\rho(\ii)=I\in\End(\mm)$
can be an $s$-representation provided the space of invariant four-forms on $\mm$
is one-dimensional. More precisely, we have the following:

\begin{epr} \label{complex} Let $\mm$ be a real representation of complex
  type of a semi-simple Lie algebra $(\hh,B_\hh)$ of compact type and consider
the representation of $\hh\oplus\uu(1)$ on $\mm$
  induced by the complex structure. 
If $\dim_\RM(\L^4\mm)^{\hh\oplus\uu(1)}=1$, then there exists a unique positive
  real number $r$ such that $\mm$ is an $s$-representation of $\hh\oplus
  \uu(1)$ with respect to the scalar product $B_\hh+rB_{\uu(1)}$. We denote here by
  $B_{\uu(1)}$ the scalar product on $\uu(1)$ satisfying $B_{\uu(1)}( \ii,\ii)=1$.
\end{epr}
\bp For every $a,b\in\hh$ we have $\tr(abI)=0$ since $a$ is skew-symmetric and $bI$ is
symmetric as endomorphisms of $\mm$. Consequently
$\tr([a,b]I)=\tr(abI)-\tr(baI)=0$.
Since $\hh$ is semi-simple we have $[\hh,\hh]=\hh$, so $\tr(aI)=0$
for all $a\in\hh$. 

Let $\omega\in\Lambda^2\mm$ be the two-form corresponding to $I$ by \eqref{o}.
An orthonormal basis of $(\hh\oplus
  \uu(1),B_\hh+rB_{\uu(1)})$ is $\{a_k,\tfrac1{\sqrt r}\ii\}$. The element in $\Lambda^2\mm$
induced by $\ii$ being $\tilde\rho(\ii)=\omega$, the image of the Casimir element corresponding to 
$B_\hh+rB_{\uu(1)}$ in $\Lambda^4\mm$ is $\tilde\rho(\Cas_\hh)+\tfrac1r\omega\wedge\omega$.
Both summands are clearly $\hh$-invariant. To see that they are $\uu(1)$-invariant, note that by \eqref{ind}, both $\omega$
and the 2-forms $\tilde a\in\Lambda^2\mm$ for $a\in\hh$ are invariant under the induced action of
$\uu(1)$ on $\Lambda^2\mm$.
The hypothesis thus shows that there exists some real constant $c$ with
\beq\label{ca}\tilde\rho(\Cas_\hh)=
c\,\omega\wedge\omega.\eeq
It remains to show that $c$ is negative (since then one can apply Proposition \ref{cas} for $r=-1/c$).

Let $\lambda:\Lambda^k\mm\to\Lambda^{k-2}\mm$ denote the metric adjoint of 
the wedge product with $\omega$. It satisfies 
$$\lambda(\tau)=\tfrac12\sum_iIe_i\lrcorner e_i\lrcorner\tau$$
for every $\tau\in\Lambda^k\mm$, where $\lrcorner$ denotes the inner product.
Let $2n\ge 4$ be the real dimension of $\mm$. Then $\lambda(\omega)=n$ and $\lambda(\omega\wedge\omega)=(2n-2)\omega$.
In terms of $\lambda$, the relation $\tr(aI)=0$ obtained above just reads 
$\lambda(\tilde a)=0$ for all $a\in\hh$. We then get
$$\lambda(\tilde\rho(\Cas_\hh))=\lambda(\sum_k(\tilde a_k\wedge\tilde a_k))=\sum_{i,k}(a_ke_i\wedge a_kIe_i),$$
whence
$$\lambda^2(\tilde\rho(\Cas_\hh))=-\frac12\sum_{i,k}B_\mm(Ia_ke_i,Ia_ke_i).$$
From \eqref{ca} we thus find $c(2n^2-2n)=-\frac12\sum_{i,k}B_\mm(Ia_ke_i,Ia_ke_i)$, so $c$ is negative.\r

We consider now the quaternionic case. It turns out that a real representation $\mm$ of quaternionic
type is never an $s$-representation. Indeed, if $\mm$ is an $s$-representation, it follows from the proof of Lemma \ref{cs} that 
the three elements $a_I,\ a_J$ and $a_K$ defined from the quaternionic structure by \eqref{daj} belong to the center of $\hh$,
so in particular $a_I$ and $a_J$ commute.
On the other hand, \eqref{aj1} shows that  $a_I$ and $a_J$ anti-commute when acting on $\mm$.

However, like in the complex case, there are situations when one may turn $\mm$ into an $s$-representation by 
adding an extra summand $\sp(1)$ to $\hh$, and making it act on $\mm$ via the quaternionic structure.

\begin{epr} \label{quaternionic}  Let $\mm$ be a real representation of quaternionic
  type of a Lie algebra $(\hh,B_\hh)$ of compact type and consider
the representation of $\hh\oplus\sp(1)$ on $\mm$
  induced by the quaternionic structure. If 
  $\dim_\RM(\L^4\mm)^{\hh\oplus\sp(1)}=1$, then  there exists a unique positive 
  real number $r$ such that the induced representation of $(\hh\oplus
  \sp(1),B_\hh+rB_{\sp(1)})$ on $\mm$ is an $s$-representation, where
  $B_{\sp(1)}$ denotes the scalar product of $\sp(1)$ such that $\ii,\jj,\kk$ is an orthonormal basis.
\end{epr}
\bp Let $\omega_I$, $\omega_J$ and $\omega_K$ denote the elements in $\Lambda^2\mm$
induced by the quaternionic structure $\{\ii,\jj,\kk\}$ via \eqref{o}.
Like before, the image of the Casimir element corresponding to 
$B_\hh+rB_{\sp(1)}$ in $\Lambda^4\mm$ is 
$$\tilde\rho(\Cas_\hh)+\tfrac1r(\omega_I\wedge\omega_I+\omega_J\wedge\omega_J+\omega_K\wedge\omega_K).$$
Both terms are clearly $\hh$-invariant by \eqref{ind}. To see that they are $\sp(1)$-invariant, we use \eqref{ind} again
to see that the induced action of $\sp(1)$ on $\Lambda^2\mm$ satisfies
$$\ii_*(\tilde a)=0\qquad \forall a\in\hh\qquad\hbox{and}\qquad \ii_*(\omega_I)=0,\  \ii_*(\omega_J)=2\omega_K,\ \ii_*(\omega_K)=-2\omega_J,$$
whence $\ii_*(\tilde\rho(\Cas_\hh))=0$ and $\ii_*(\omega_I\wedge\omega_I+\omega_J\wedge\omega_J+\omega_K\wedge\omega_K)=
4\omega_J\wedge\omega_K-4\omega_K\wedge\omega_J=0.$ The invariance with respect to $\jj_*$ and $\kk_*$ can be
proved in the same way.
The hypothesis thus shows that there exists some real constant $c$ with
\beq\label{ca1}\tilde\rho(\Cas_\hh)=c(\omega_I\wedge\omega_I+\omega_J\wedge\omega_J+\omega_K\wedge\omega_K),\eeq
and again it remains to show that $c$ is negative.

Let $\lambda_\ii:\Lambda^k\mm\to\Lambda^{k-2}\mm$ denote the metric adjoint of 
the wedge product with $\omega_I$. From the computations in the complex case we have 
$$\lambda_\ii^2(\tilde\rho(\Cas_\hh))=-\frac12\sum_{i,k}B_\mm(Ia_ke_i,Ia_ke_i)\qquad\hbox{and}\qquad 
\lambda_\ii^2(\omega_I\wedge\omega_I)=2n(n-1),$$
where $2n$ denotes the real dimension of $\mm$. An easy computation gives 
$$\lambda_\ii^2(\omega_J\wedge\omega_J)=\lambda_\ii^2(\omega_K\wedge\omega_K)=2n,$$
so from \eqref{ca1} we get $c(2n^2+2n)=-\frac12\sum_{i,k}B_\mm(Ia_ke_i,Ia_ke_i)$, showing that $c$ is negative.\r

We can summarize Corollary \ref{c1} and Propositions \ref{complex}, \ref{quaternionic} by saying that a certain 
condition on the invariant part
of $\Lambda^4 \mm$ is sufficient for the existence of an $s$-representation on $\mm$. Conversely one might
ask whether this condition is also necessary for a given $s$-representation. It turns out that this is always
the case if $\hh$ is simple. More precisely, we have:

\begin{epr}\label{rem}
Let $(\hh,B_\hh)$ be a simple Lie algebra of compact type and $\mm$ an irreducible 
representation of $\hh$  over $\RM$.
\begin{enumerate}
\item
If $\mm$ is an $s$-representation representation
  of  $\hh$, then $(\L^4\mm)^\hh=0$.
\item If $\mm$ has complex type and is an
  $s$-representation 
  of $(\hh\oplus \uu(1),B_\hh+rB_{\uu(1)})$ for some positive real number
  $r$, then $\dim_\RM(\L^4\mm)^{\hh \oplus \uu(1)}=1$. 
\item If $\mm$ has quaternionic type and is
  an $s$-representation 
  of $(\hh\oplus \sp(1),B_\hh+rB_{\sp(1)})$ for some positive real number
  $r$, then $\dim_\RM(\L^4\mm)^{\hh\oplus\sp(1)}=1$. 
\end{enumerate}
\end{epr}

\bp The statement of (1) is already contained in \cite{ko56}. For the proof one has to use the well-known fact that
the dimension of $(\Lambda^4 \mm)^\hh$ is just the fourth Betti number of the corresponding
symmetric space $G/H$. Now, since in the case of compact symmetric spaces the Poincar\'e polynomials
are known explicitly, it is easy to check that $\bb_4(G/H)$ vanishes. The proof in the cases (2) and (3) is similar, and is
based on the computation of the fourth Betti numbers of Hermitian symmetric spaces and Wolf spaces.
\r

\section{Applications to complex representations}

In this section we will give some applications of our characterization of $s$-representations in
order to classify complex representations whose exterior powers have certain irreducibility 
properties. From now on $\mm$ will denote a {\em complex} irreducible representation of some
Lie algebra $\hh$ of compact type. We will study three instances, corresponding to the cases where 
$\mm$ has a real structure, $\mm$ is purely complex, or $\mm$ has a quaternionic structure.

\subsection{Representations with a real structure} Our main result in this case is the classification
of all complex representations $\mm$ with real structure whose fourth exterior power has no trivial summand.
Let $\llbracket \mm\rrbracket$ denote the real part of $\mm$, {\em i.e.} the 
fix point set of the real structure. Then $\llbracket \mm\rrbracket$ is a real representation of $\hh$ and $\mm=\llbracket 
\mm\rrbracket\otimes_\RM\CM$.

\begin{ath}\label{l4}
Let $\mm$ be a complex irreducible faithful representation with real structure of a Lie algebra $\hh$ of compact type such that 
$(\Lambda^4 \mm)^\hh = 0$. Then $\hh$ is simple and the pair $(\hh, \llbracket \mm\rrbracket)$ belongs to the following list:
\begin{center}
\begin{tabular}{|l|l|l|}
\hline
Helgason's type &  $\hh$   &     $\llbracket \mm\rrbracket$   \\
\hline \hline
& $\hh$   &      $\hh$    \\
\hline
\rm{BD I} & $\so(n), n \neq 4$ & $\RM^n$  \\
\hline
\rm{A I} & $\so(n), n \neq 4$ & $\Sym^2_0 \, \RM^n$  \\
\hline
\rm{A II} & $\sp(n)$ & $\Lambda^2_0\,  \HM^n$  \\
\hline
\rm{F II} & $\spin(9)$ & $\Sigma_9$  \\
\hline
\rm{E I} & $\sp(4)$ & $\Lambda^4_0 \, \HM^4$  \\
\hline
\rm{E IV} & $F_4$     &   $V_{26}$  \\
\hline
\rm{E V} & $\su(8)$ &  $\llbracket\Lambda^4 \CM^8\rrbracket$   \\
\hline
\rm{E VIII} & $\spin(16)$ &  $\Sigma^+_{16}$  \\
\hline
\end{tabular}
\end{center}
where $V_{26}$ denotes the real $26$-dimensional irreducible representation of $F_4$ and
$\llbracket \mm\rrbracket = \hh$ in the first row denotes the adjoint representation of $\hh$.
\end{ath}
\begin{proof} Since $(\Lambda^4 \mm)^\hh=(\Lambda^4_\RM\llbracket \mm\rrbracket)^\hh\otimes\CM$, the
hypothesis implies that $(\Lambda^4_\RM\llbracket \mm\rrbracket)^\hh=0$.
From Lemma~\ref{tensor} it follows that $\hh$ has to be simple, and Corollary~\ref{c1} 
implies that $\llbracket \mm\rrbracket$ is an $s$-representation
and thus $\llbracket \mm\rrbracket$ is  the isotropy representation of an irreducible symmetric space $G/H$ of compact type
with $H$ simple.

The list of possible pairs $(\hh, \llbracket \mm\rrbracket)$ then follows from the list of irreducible symmetric spaces of compact type
\cite[p. 312-314]{besse}. Here the adjoint representation on $\llbracket \mm\rrbracket = \hh$ corresponds to
the isotropy representation on symmetric spaces of the type II, {\em i.e.} of the form $(H\times H)/ H$.

Conversely, the fourth exterior power of the representations in the table above have no invariant elements 
by Proposition \ref{rem}. In some cases 
a direct proof can also be given, see Proposition \ref{41} below.
\end{proof}

\begin{ere}\label{outer}
We will see later on in Proposition \ref{spin} that the
real half-spin representation $\Sigma_8^+$ is also an $s$-representation, and has no invariant elements in 
its fourth exterior power ({\em cf.} also Proposition \ref{rem}). One may thus wonder why it does not 
appear in the above table. The explanation is that it actually appears in a disguised form, as the standard
representation of $\so(8)$ on $\RM^8$. To make this more precise, note that $n$-dimensional representations 
are usually classified up to isomorphism,
{\em i.e.} up to composition with some element in the inner automorphism group of $\so(n)$. On the other hand, 
if one wants to classify all pairs $(\hh,\mm)$ with $(\Lambda^4\mm)^\hh=0$, then there is another group acting 
on the space of solutions: the outer automorphism group of $\so(n)$.
Our classification above is up to the action of this group.
In particular, the triality phenomenon in dimension 8 can be interpreted by saying that the outer automorphism 
group of $\so(8)$ is isomorphic to the permutation group of the set $\{\RM^8,\Sigma_8^+,\Sigma_8^-\}$ of 
8-dimensional representations of $\so(8)\cong\spin(8)$. The same remark also applies below to complex representations, 
where taking the conjugate of a representation can be viewed as composing it with the non-trivial outer automorphism of $\su(n)$. 
\end{ere}

\subsection{Representations with irreducible second exterior power} As another application of the above ideas, we will
now obtain in a simple geometrical way Dynkin's classification \cite[Thm. 4.7]{dynkin} of complex
representations $\mm$ with $\Lambda^2\mm$ irreducible.

\begin{ath}\label{com}
Let $\mm$ be a complex irreducible faithful representation of a Lie algebra $\hh$ of compact type such that
$\Lambda^2 \mm$ is irreducible. Then either $\hh = \hh_0$ is simple, or
$\hh = \hh_0 \oplus \uu(1)$ with $\hh_0$ simple, and the pair $(\hh_0, \mm)$ belongs to the following list:

\begin{center}
\begin{tabular}{|l|l|l|}
\hline
Helgason's type & $\hh_0$   &     $\mm$  \\
\hline \hline
\rm{A III} & $\su(n)$  &      $\CM^n$   \\
\hline
\rm{D III} & $\su(n)$  &     $ \Lambda^2\CM^n$    \\
\hline
\rm{C I} &  $\su(n)$  &     $\Sym^2 \CM^n$    \\
\hline
\rm{BD I} & $\so(n), n \neq 4$ &  $\RM^n \otimes \CM$  \\
\hline
\rm{E III} & $\spin(10)$ & $\Sigma_{10}$  \\
\hline
\rm{E VII} & $\E_6$ & $V_{27}$  \\
\hline
\end{tabular}
\end{center}
where $V_{27}$ denotes the $27$-dimensional irreducible representation of $\E_6$.
\end{ath}
\begin{proof} 
If $\hh$ is not simple, it may be written as the sum
of two ideals, $\hh = \hh_0 \oplus \hh_1$, and $\mm$ is the tensor product
representation $\mm = E \otimes F$. It follows that
$$
\Lambda^2 \mm = \Lambda^2 (E \otimes F) 
\cong 
(\Lambda^2 E \otimes \Sym^2F)\,    \oplus  \, (\Sym^2 E \otimes \Lambda^2 F).
$$
Hence $\Lambda^2 \mm$ can only be irreducible if one factor, say $F$, is
one-dimensional. Since $F$ is a faithful representation, one must have $\hh_1 = \uu(1)$. 
This argument shows that every ideal of $\hh$ has either dimension or co-dimension at most one.
In particular, $\hh_0$ is simple.

Consider now the real representation $\mm^\RM$ of $\hh_0$ obtained by forgetting the complex multiplication in $\mm$. 
The Lie algebra $\uu(1)$ acts on the fourth exterior power $\Lambda^4_\RM \mm^\RM$ by extending the action of the complex structure $J$
from $\mm^\RM$.
We claim that 
the space of invariant elements $ (\Lambda_\RM^4 \mm^\RM)^{\hh_0 \oplus \uu(1)}$ is one-dimensional. 
If we denote as usual by $\Lambda^{p,q}\mm:=\Lambda^p\mm\otimes\Lambda^q\bar \mm$, then
$$
\Lambda_\RM^4 \mm^\RM = 
\llbracket \Lambda^{4,0} \mm \rrbracket  \oplus 
\llbracket \Lambda^{3,1} \mm \rrbracket \oplus
\llbracket \Lambda^{2,2} \mm \rrbracket.
$$ 
Since $J^2$ acts as $-(p-q)^2\id$ on 
$\llbracket \Lambda^{p,q}\mm \rrbracket$, it follows that the  
$\uu(1)$-invariant part of $\Lambda^4_\RM \mm^\RM$ is the third summand $\llbracket \Lambda^{2,2} \mm \rrbracket 
= \llbracket    \Lambda^2 \mm \otimes \Lambda^2 \bar{ \mm}\rrbracket
= \llbracket \End(\Lambda^2 \mm) \rrbracket$. Consequently, 
\beq\label{u1} (\Lambda_\RM^4 \mm^\RM)^{\hh_0 \oplus \uu(1)}
= \llbracket \End(\Lambda^2 \mm) \rrbracket^{\hh_0}=\llbracket (\End(\Lambda^2 \mm))^{\hh_0} \rrbracket\eeq
is one-dimensional since by assumption $\Lambda^2 \mm$
is irreducible as representation of $\hh$, so also of $\hh_0$. We can therefore 
apply Proposition~\ref{complex} to realize $\mm^\RM$ as an $s$-representation of $\hh_0\oplus\uu(1)$, 
so $\mm^\RM$ is the isotropy representation of some Hermitian symmetric space. Checking again the list in \cite[pp. 312-314]{besse} we
obtain the possible pairs $(\hh_0, \mm)$ as stated above.

Conversely, if $(\hh_0,\mm)$ belongs to the above list, then $ (\Lambda_\RM^4 \mm^\RM)^{\hh_0 \oplus \uu(1)}$ is one-dimensional 
by Proposition \ref{rem}, thus \eqref{u1} shows that $\Lambda^2 \mm$ is irreducible.
\end{proof}

\subsection{Representations with quaternionic structure}
As another application we will now consider complex representations $\mm$
of $\hh$ with quaternionic structure. Such representations can be characterized by the existence of an invariant 
element in $\Lambda^2 \mm$, which is therefore never irreducible. Considering the $\hh$-invariant
decomposition $\Lambda^2 \mm = \Lambda^2_0\mm \oplus \CM$, one can nevertheless ask whether $\Lambda^2_0 \mm$
can be irreducible. The classification of such representations is given by the following:

\begin{ath}\label{33}
Let $\mm$ be a complex irreducible faithful representation of a Lie algebra $\hh$ of compact type with 
a quaternionic structure, and let
$\Lambda^2\mm = \Lambda^2_0\mm \oplus \CM$ be the standard decomposition of the second exterior
power of $\mm$. If  the $\hh$-representation
$\Lambda^2_0\mm$ is irreducible then $\hh$ is simple and the pair
$(\hh, \mm)$ belongs to the following list:
\begin{center}
\begin{tabular}{|l|l|l|}
\hline
Helgason's type & $\hh$   &     $\mm$   \\
\hline \hline
\rm{C II} & $\sp(n)$ & $\HM^n$  \\
\hline
\rm{F I} & $\sp(3)$   &      $ \Lambda^3_0 \HM^3$    \\
\hline
\rm{G I} & $\sp(1)$ &  $\Sym^3 \HM$  \\
\hline
\rm{E II} & $\su(6)$ & $\Lambda^3 \CM^6$  \\
\hline
\rm{E VI} & $\spin(12)$ & $\Sigma_{12}^+$  \\
\hline
\rm{E IX} & $\E_7$ & $V_{56}$  \\
\hline
\end{tabular}
\end{center}
where $V_{56}$ is the $56$-dimensional irreducible representation of $\E_7$.
\end{ath}
\begin{proof} Let $\ii$ denote the complex structure of $\mm$ and let $\jj$ be the quaternionic structure, {\em i.e.} a real 
endomorphism of $\mm^\RM$ anti-commuting with $\ii$ and satisfying $\jj^2=-\id$.
Like before, if $\hh$ is not simple, one can write $\hh = \hh_0 \oplus \hh_1$, $\mm =E \otimes F$ and
$$
\Lambda^2 \mm 
\cong 
(\Lambda^2 E \otimes \Sym^2F)\,    \oplus  \, (\Sym^2 E \otimes \Lambda^2 F).
$$
If $E$ and $F$ have both dimension larger than one, then both summands in the above expression have the same property, which is 
impossible because of the hypothesis. Assume that one factor, say $F$, is
one-dimensional. Since $F$ is a faithful representation, one must have $\hh_1 = \uu(1)$. Let $a\ne 0$
be the endomorphism of $\mm$ determined by the generator of $\uu(1)$. By the Schur Lemma there exists $z\in\CM^*$ such that $a=z\,\id$. 
Since $a$ commutes with the quaternionic structure we must have $z\in\RM$, which is impossible since $a$ has to be a skew-symmetric
endomorphism of $\mm^\RM$ with respect to some $\hh$-invariant scalar product. Thus $\hh$ is simple.

The Lie algebra $\sp(1)$ acts on the fourth exterior power $\Lambda^4_\RM \mm^\RM$ by extending the action of the (real) endomorphisms
$\ii$ and $\jj$ of $\mm^\RM$.
We claim that 
the space of invariant elements $(\Lambda_\RM^4 \mm^\RM)^{\hh \oplus \sp(1)}$ is one-dimensional. Using \eqref{u1} we see that
$$ (\Lambda_\RM^4 \mm^\RM)^{\hh \oplus \uu(1)}
=\llbracket( \End(\Lambda^2 \mm))^{\hh} \rrbracket
=\llbracket (\End(\Lambda^2_0 \mm\oplus\CM))^{\hh} \rrbracket=\llbracket (\End(\Lambda^2_0 \mm))^{\hh} \rrbracket\oplus \RM$$
is two-dimensional since by assumption $\Lambda^2_0 \mm$
is irreducible. The first summand is generated by $\Omega_1:=\omega_I\wedge\omega_I$, whereas the second one is generated
by $\Omega_2:=\omega_J\wedge\omega_J+\omega_K\wedge\omega_K$. Using \eqref{ind} we readily obtain 
$\jj_*\Omega_1=-4\omega_K\wedge\omega_I$ and $\jj_*\Omega_2=4\omega_I\wedge\omega_K$, thus showing that 
$(\Lambda_\RM^4 \mm^\RM)^{\hh \oplus \sp(1)}$ is one-dimensional and spanned by $\Omega_1+\Omega_2$. 
We can therefore 
apply Proposition~\ref{complex} to realize $\mm^\RM$ as an $s$-representation of $\hh\oplus\sp(1)$.
Consequently, $\mm$ is the isotropy representation of a Wolf space, and thus belongs to the above table by \cite[pp. 312-314]{besse}.

Conversely, it is standard fact that $\Lambda^2_0\HM^n$ is an irreducible $\sp(n)$-representation, and
one can check ({\em e.g.} using the LiE software \cite{lie}) that for all other representations $\mm$ in this table,
$\Lambda^2_0\mm$ is indeed irreducible. 
\end{proof}

\section{Spin representations and exceptional Lie algebras}

In this section we obtain a completely self-contained
construction of exceptional simple Lie algebras based on the results in Section 2. We will only give the details for the construction of $\E_8$
arising from the half-spin representation of $\Spin(16)$, since all the other exceptional simple Lie algebras
can be constructed by similar methods using spin representations. Conversely, we give a short algebraic argument 
showing that there are no other spin representations which are $s$-representations 
but those giving rise to exceptional Lie algebras.

\subsection{A computation-free argument for the existence of $\E_8$}
As already mentioned in the introduction, the only non-trivial part in the construction of $\E_8$
is to check that the natural bracket on $\spin(16)\oplus\Sigma_{16}^+$ constructed 
in Lemma \ref{l1} satisfies the Jacobi identity. This follows directly from Corollary \ref{c1}, together with the following:

\begin{epr}\label{41}
The fourth exterior power of the real half-spin representation $\Sigma_{16}^+$ has no trivial summand.
\end{epr}
\bp
One can use the plethysm function of the LiE software \cite{lie} to check that the fourth exterior power of $\Sigma_{16}^+$ 
has nine irreducible summands, each of them being non-trivial. However, our purpose is exactly to replace such 
brute force computations by conceptual arguments!

Let $\langle.,.\rangle_\S$ and $\langle.,.\rangle_\Sigma$ be $\Spin(16)$-invariant scalar products 
on $\spin(16)$ and $\Sigma_{16}^+$ respectively.
We start by recalling that the second exterior power of the real half-spin representation in dimension $8k$ decomposes
in irreducible summands as
$$\Lambda^2(\Sigma_{8k}^+)\simeq\bigoplus_{i=1}^k\Lambda^{4i-2}(\RM^{8k}).$$
This isomorphism can also be proved in an elementary way. Indeed, the right hand term
acts skew-symmetrically and faithfully by Clifford multiplication on $\Sigma_{8k}^+$
and thus can be identified with a sub-representation of $\Lambda^2(\Sigma_{8k}^+)$.
On the other hand, its dimension is equal to
$$\dim\bigoplus_{i=1}^k\Lambda^{4i-2}(\RM^{8k})=\frac18\left(2^{8k}-(1+i)^{8k}-(1-i)^{8k}\right)=
2^{4k-2}(2^{4k-1}-1)=\dim\,\Lambda^2(\Sigma_{8k}^+).$$
For $k=2$ we thus get 
\beq \label{s16}\Lambda^2(\Sigma_{16}^+)\simeq \Lambda^2(\RM^{16})\oplus\Lambda^6(\RM^{16}).\eeq
Recall the standard decomposition
$$\Sym^2(\Lambda^2(\Sigma_{16}^+))\simeq\R\oplus\Lambda^4(\Sigma_{16}^+),$$
where $\R$ is the kernel of the Bianchi map $\b:\Sym^2(\Lambda^2(\Sigma_{16}^+))\to\Lambda^4(\Sigma_{16}^+).$
The trace element $R_1\in\Sym^2(\Lambda^2(\Sigma_{16}^+))$
defined by 
$$R_1(v, w,v', w'):=\langle v\wedge w,v'\wedge w'\rangle_\Sigma $$
is invariant under the action of $\Spin(16)$ and belongs to $\R$ since $\b(R_1)=0$.

Assume for a contradiction that $\Lambda^4(\Sigma_{16}^+)$ contains some invariant element $\Omega$ and consider the invariant 
element 
$R_2\in\Sym^2(\Lambda^2(\Sigma_{16}^+))$
defined by 
$$R_2(v, w,v', w'):=\langle [v,w],[v',w']\rangle_\S ,$$
where $[.,.]$ is the bracket defined by Lemma \ref{l1}. 
Since the two irreducible summands in \eqref{s16} are not isomorphic, the space of invariant elements in 
$\Sym^2(\Lambda^2(\Sigma_{16}^+))$ has dimension two. Hence there exist 
real constants $k,l$ such that $R_2=kR_1+l\Omega$. In particular we would have
\beq\label{a}|[v,w]|_\S^2=k|v\wedge w|_\Sigma^2,\qquad\forall v,w\in\Sigma^+_{16}.\eeq
Since $\dim(\spin(16))=120$ is strictly smaller than $\dim(\Sigma^+_{16})-1=127$, one can find non-zero vectors
$v_0,w_0\in \Sigma^+_{16}$ such that $v_0\wedge w_0\ne 0$ and $\langle v_0,aw_0\rangle_\Sigma =0$ for all $a\in\spin(16)$. 
By the definition of the bracket in Lemma \ref{l1} (4), this implies
$[v_0,w_0]=0$, so using \eqref{a} for $v=v_0$ and $w=w_0$ yields $k=0$. By  \eqref{a} again, this would imply 
$[v,w]=0$ for all $v,w\in\Sigma^+_{16}$, so we would have
$$0=\langle a,[v,w]\rangle_\S =\langle av,w\rangle_\Sigma ,\qquad \forall a\in\spin(16),\ \forall v,w\in\Sigma^+_{16},$$
which is clearly a contradiction.
\r

\subsection{The construction of $\F_4$, $\E_6$ and $\E_7$}
Consider the following spin representations: $\Sigma_9$, which is real, $\Sigma_{10}$ which is purely 
complex, and $\Sigma_{12}^+$ which is quaternionic. 
In order to show that they give rise to $s$-representations of $\spin(9)$, $\spin(10)\oplus\uu(1)$ and 
$\spin(12)\oplus\sp(1)$ respectively,
we need to check that one can apply the criteria in Corollary~\ref{c1}, Proposition~\ref{complex} and 
Proposition~\ref{quaternionic}. 
Taking into account the results in Section 3, it suffices to show that $(\Lambda^4\Sigma_9)^{\spin(9)}=0$, 
and that $\Lambda^2\Sigma_{10}$  and $\Lambda^2_0\Sigma_{12}^+$ are irreducible. The first assertion can 
be proved like in Proposition \ref{41}, whereas the 
two other follow from the classical decompositions of the second exterior power of spin representations
$$\Lambda^2\Sigma_{10}\cong \Lambda^3(\CM^{10}),\qquad \Lambda^2\Sigma_{12}^+\cong \Lambda^0(\CM^{12})\oplus \Lambda^4(\CM^{12}).$$

\subsection{On spin representations of Lie type}

In this final part we will show that very few spin representations are of Lie type. To make things
precise, recall that the real Clifford algebras $\Cl_n$ are of the form 
$\KM(r)$ or $\KM(r)\oplus \KM(r)$ where:
\begin{center}
\begin{tabular}{|c|c|c|c|c|c|c|c|c|}
\hline
$n:$ & $8k+1$   &    $8k+2$   & $8k+3$   & $8k+4$   & $8k+5$   & $8k+6$   & $8k+7$   & $8k+8$     \\
\hline
$r:$ & $2^{4k}$ & $2^{4k}$ & $2^{4k}$ & $2^{4k+1}$ & $2^{4k+2}$ & $2^{4k+3}$ & $2^{4k+3}$ & $2^{4k+4}$\\
 \hline
$\Cl_n:$ & $\CM(r)$ & $\HM(r)$ & $\HM(r)\oplus \HM(r)$ & $\HM(r)$ & $\CM(r)$ & $\RM(r)$ & $\RM(r)\oplus \RM(r)$ & $\RM(r)$  \\
\hline
\end{tabular}
\end{center}
The Clifford representation of the real Clifford algebra is by definition the unique irreducible representation 
of $\Cl_n$ for $n\ne 3\mod 4$, and the direct sum of the two inequivalent representations
for $n= 3\mod 4$. The real spinor representation $\Sigma_n$ is the restriction of the Clifford representation to $\spin(n)\subset\Cl_{n-1}$
(note the shift from $n$ to $n-1$). For $n\ne 0\mod 4$ the spin representation is irreducible, and for $n=0\mod 4$ it decomposes
as the direct sum of two irreducible representations $\Sigma_n=\Sigma^+_n\oplus\Sigma^-_n$. We introduce the notation
$$\Sigma_n^{(+)}:=\begin{cases}\Sigma_n^+& \hbox{if}\  n=0\mod 4,\\
                        \Sigma_n& \hbox{if}\  n\ne 0\mod 4.\end{cases}$$ 
The table above shows that 
the spin representation $\Sigma_n^{(+)}$ is of real type for $n=0,1,7$ mod 8, of complex type for $n=2$ or $6$ mod 8 and of quaternionic
type for $n=3,4,5$ mod 8. We define the Lie algebras 
$$\widetilde\spin(n):=\begin{cases}\spin(n) &\hbox{if}\ n=0,1,7 \mod 8,\\  
\spin(n)\oplus \uu(1) &\hbox{if}\ n=2\ \hbox{or}\ 6 \mod 8,\\
\spin(n)\oplus \sp(1) &\hbox{if}\ n=3,4,5 \mod 8. \end{cases}$$ 
We can view $\Sigma_n^{(+)}$ as a $\widetilde\spin(n)$-representation, where
the $\uu(1)$ or $\sp(1)$ actions are induced by the complex or quaternionic structure
of the spin representation in the last two cases.

We study the following question: {\em for which $n\ge 5$ is $\Sigma_n^{(+)}$ a representation of Lie type of $\widetilde\spin(n)$?}
We will see that there are almost no other examples but 
the above examples which lead to the construction of exceptional 
Lie algebras. This is a consequence of
the very special structure of the weights of the spin representations (see also \cite{mihaela}, \cite{cliff} for a much more general approach
to this question). 

\begin{epr}\label{spin} For $n\ge 5$,  the representation $\Sigma_n^{(+)}$ of $\widetilde\spin(n)$
is of Lie type if and only if $n\in\{5,6,8,9,10,12,16\}$.
\end{epr}
\bp
The representation $\Sigma_n^{(+)}$ of $\widetilde\spin(n)$ is of Lie type if and only if 
there exists a Lie algebra structure on $\gg:=\widetilde\spin(n)\oplus\Sigma_n^{(+)}$ satisfying conditions (1), (2) and (4) in Lemma \ref{l1}
with respect to some $\ad_{\widetilde\spin(n)}$ invariant scalar products on $\widetilde\spin(n)$ and $\Sigma_n^{(+)}$. 
We will always consider some fixed Cartan subalgebra of $\widetilde\spin(n)$, which is automatically a Cartan subalgebra of
$\gg$ since the (half-)spin representations have no zero weight.

Consider first the case $n=8k$. Since $\widetilde\spin(8k)=\spin(8k)$, the scalar products above are unique up to some constant.
We choose the scalar product $\la.,.\ra$ on $\spin(8k)$ such that 
in some orthonormal basis $\{e_1,\ldots,e_{4k}\}$ of the Cartan subalgebra of $\spin(8k)$, the roots of $\spin(8k)$ are
$$\R=\{\pm e_i\pm e_j\ |\ 1\le i<j\le 4k\},$$
and the weights of the (complexified) half-spin representation $\Sigma_{8k}^+\otimes\CM$ are
$$\W=\{\tfrac12\sum_{i=1}^{4k}\e_ie_i\ |\ \e_i=\pm1, \
\e_1\ldots\e_{4k}=1\}.$$
The union $\R\cup\W$ is then the root system
of $\gg$, which is a Lie algebra of compact type. In particular, the quotient  
\beq\label{q}q(\a,\b):=\frac{2\la \a,\b\ra}{\la \b,\b\ra}\eeq
is an integer satisfying $|q(\a,\b)|\le 3$ for all $\a,\b\in
\R\cup\W$ (cf. \cite{adams}, p. 119). Taking $\a=e_1+e_2$ and
$\b=\tfrac12\sum_{i=1}^{4k}e_i$ we 
get $q(\a,\b)=2/k$, whence $k=1$ or $k=2$, so $n=8$ or $n=16$. Conversely, the real half-spin representations 
$\Sigma^+_8$ and $\Sigma^+_{16}$ are of Lie type (actually they are $s$-representations with augmented Lie algebras
$\spin(9)=\spin(8)\oplus\Sigma^+_8$ and 
$\ee_8=\spin(16)\oplus\Sigma^+_{16}$).

If $n=8k-1$, a similar argument using the root $\a=e_1$ of $\spin(8k-1)$ and the weight $\b=\tfrac12\sum_{i=1}^{4k-1}e_i$ 
of the spin representation shows that $q(\a,\b)=2/(4k-1)$ cannot be an integer. 

If $n=8k+1$, one has  $q(\a,\b)=1/k$ for $\a=e_1$ and $\b=\tfrac12\sum_{i=1}^{4k}e_i$, so $k=1$. 
Conversely, $\Sigma_9$ is an $s$-representation, as shown by the exceptional Lie algebra
$\ff_4=\spin(9)\oplus\Sigma_9$.

Consider now the case when the spin representation is complex, {\em i.e.} $n=4k+2$, with $k\ge 1$. Assume that on $\gg:=(\spin(4k+2)\oplus \uu(1))\oplus\Sigma_{4k+2}$ there exists a Lie algebra structure
satisfying conditions (1), (2) and (4) in Lemma \ref{l1}
with respect to some $\ad_{\spin(4k+2)\oplus \uu(1)}$ invariant scalar products on $\spin(4k+2)\oplus \uu(1)$ and $\Sigma_{4k+2}$.
The latter scalar product is defined up to a scalar, whereas for the first one there is a two-parameter family of possible choices. By rescaling, 
we may assume that the restriction of the scalar product on the $\spin(4k+2)$ summand is such that in some orthonormal basis $\{e_1,\ldots,e_{2k+1}\}$ 
of the Cartan subalgebra, the root system of $\spin(4k+2)$ is 
$$\R=\{\pm e_i\pm e_j\ |\ 1\le i<j\le 2k+1\}.$$
There exists a unique vector $e_{2k+2}\in\uu(1)$ such that  the set of weights of the representation  $\Sigma_{4k+2}\otimes\CM\cong\Sigma_{4k+2}\oplus\overline{\Sigma}_{4k+2}$ of $\spin(4k+2)\oplus \uu(1)$ is
$$\W=\{\tfrac12\sum_{i=1}^{2k+2}\e_ie_i\ |\ \e_i=\pm1, \
\e_1\ldots\e_{2k+2}=1\}.$$
We denote by $x:=|e_{2k+2}|^2$ its square norm. The root system of $\gg$ is clearly $\R(\gg)=\R\cup\W$.

Recall that for every non-orthogonal roots $\a$ and $\b$ of $\gg$, their sum or difference is again a root \cite{adams}.
On the other hand, neither the sum, nor the difference of the two roots 
$\a:=\tfrac12(\sum_{i=1}^{2k+2}e_i)$ and $\b:=\tfrac12(\sum_{i=1}^{2k}e_i-e_{2k+1}-e_{2k+2})$ of $\gg$
belongs to $\R(\gg)=\R\cup\W$. Thus $\la\a,\b\ra=0$, which implies $x=2k-1$. 
Consider now the root $\gamma:=e_1+e_2$. The integer defined in \eqref{q} is 
$$q(\gamma,\a)=\frac{2\la \gamma,\a\ra}{\la \a,\a\ra}=\frac{2}{\tfrac14(2k+1+x)}=\frac2k,$$
showing that necessarily $k=1$ or $k=2$. Conversely, both cases do occur, since $\Sigma_6$ and $\Sigma_{10}$ are $s$-representations 
of $\spin(6)\oplus\uu(1)\cong\uu(4)$ and $\spin(10)\oplus\uu(1)$ with augmented Lie algebras $\uu(5)$ and $\ee_6$ respectively.
 
Similar arguments (see also \cite{mihaela}) show that in the quaternionic case (when $n=3,4,5 \mod 8$) there are only two representations 
$\Sigma_n^{(+)}$ of $\spin(n)\oplus\sp(1)$ which are of Lie type, for $n=5$ and $n=12$. They are both $s$-representations and their 
augmented Lie algebras are $\spin(5)\oplus\sp(1)\oplus\Sigma_5\cong\sp(2)\oplus\sp(1)\oplus\HM^2\cong\sp(3)$ and 
$\spin(12)\oplus\sp(1)\oplus\Sigma_{12}^+\cong\ee_7$.\r

Note that  J. Figueroa-O'Farrill has recently asked in \cite[p. 673]{of08} about the existence of
Killing superalgebra structures on spheres other that $\SM^7$, $\SM^8$ and $\SM^{15}$.
Part I. in Proposition~\ref{spin} can be interpreted as a negative answer to this question.


\labelsep .5cm

\end{document}